\documentclass{elsarticle}
\usepackage{amssymb}
\usepackage{amsmath}
\usepackage{amsthm}
\usepackage{bm}
\usepackage{hyperref}
\usepackage{enumerate}
\usepackage{mathtools}
\usepackage{url}
\usepackage[T1]{fontenc}

\newtheorem{Theorem} {Theorem} [section]
\newtheorem{Proposition} [Theorem] {Proposition}
\newtheorem{Lemma} [Theorem] {Lemma}
\newtheorem{Corollary} [Theorem] {Corollary}

\newtheorem{Problem} [Theorem] {Problem}

\newtheorem{Conjecture} [Theorem] {Conjecture}
\newtheorem{Question} [Theorem] {Question}

\newcommand{\Ff}{{\mathbb F}}

\newcommand{\cI}{{\mathcal I}}
\newcommand{\cJ}{{\mathcal J}}
\newcommand{\cL}{{\mathcal L}}
\newcommand{\cO}{{\mathcal O}}
\newcommand{\cP}{{\mathcal P}}

\newcommand{\cS}{{\mathcal S}}

\newcommand{\PG}{\mathrm{PG}}

\newcommand{\Sp}{\mathrm{Sp}}

\newcommand{\<}{\langle}
\renewcommand{\>}{\rangle} 
\renewcommand{\phi}{\varphi} 

\makeatletter
\def\ps@pprintTitle{%
  \let\@oddhead\@empty
  \let\@evenhead\@empty
  \def\@oddfoot{\reset@font\hfil\thepage\hfil}
  \let\@evenfoot\@oddfoot
}
\makeatother

\title{Regular Sets of Lines in Rank 3 Polar Spaces}
\author[1]{Ferdinand Ihringer\fnref{fn1}\fnref{fn2}}
\ead{ferdinand.ihringer@gmail.com}
\affiliation[1]{
 organization={Department of Mathematics,
   Southern University of Science and Technology (SUSTech)},
 city={Shenzhen},
 country={China}
}
\author[2]{Morgan Rodgers\corref{cor1}\fnref{fn3}}
\ead{morgan.rodgers@rptu.de}
\affiliation[2]{
organization={Department of Mathematics, RPTU Kaiserslautern-Landau},
addressline={Gottlieb-Daimler-Stra{\ss}e 48},
city={Kaiserslautern},
postcode={67663},
country={Germany}
}

\cortext[cor1]{Corresponding author}
\fntext[fn1]{The first author has been supported by a
 postdoctoral fellowship of the Research Foundation -- Flanders (FWO) while conducting this research.}
\fntext[fn1]{The first author thanks the second author for hosting him
 in Kaiserslautern.}
\fntext[fn3]{The second author is supported by the SFB-TRR 195 `Symbolic Tools in Mathematics and their
 Application' of the German Research Foundation (DFG).}

\begin{document}

\begin{abstract}
 There are 6 families of finite polar spaces of rank $3$.
 The set of lines in a rank $3$ polar space
 form a rank $5$ association scheme.
 We determine the regular sets of minimal size
 in several of these polar spaces, and describe some interesting examples.
 We also give a new family of Cameron--Liebler sets of generators
 in the polar spaces $O^+(10,q)$ when $q = 3^h$
 using a regular set of lines in $O(7,q)$.
\end{abstract}


\maketitle

\section{Introduction}

A \textit{regular set} or \textit{equitable bipartition}
in a (finite simple) graph is a set of vertices $Y$
such that there exist constants $a$ and $b$ for which
each vertex in $Y$ has $a$ neighbors in $Y$,
while each vertex not in $Y$ has $b$ neighbors in $Y$.
Such sets can only exist for graphs which are either biregular or regular.
If $Y$ is a regular set, then $a-b$ is an eigenvalue of
the adjacency matrix of the corresponding graph~\cite[\S1.1.13]{BvM}.
We call $Y$ a \textit{regular set} of an association scheme, cf.~\cite[\S2]{BCN},
if $Y$ is a regular set for all the graphs in the association scheme.
Regular sets appear in many geometric and combinatorial contexts, where examples include
\textit{Cameron-Liebler classes},
\textit{intriguing sets}, and \textit{hemisystems} in finite geometry.
Regular sets are also closely related to the concepts of
\textit{Boolean degree $d$ functions}~\cite{FI2019} and
\textit{graphical designs}~\cite{Steinerberger2020}.
In recent years, regular sets in graphs coming from
finite geometries have received
a significant amount of attention, for instance,
see~\cite{BKLP2007,BL2023,BLP2009,CP2019,DBDMR2016,DBDIM2022,FLX2023,
FMRXZ2021,FMX2015,FX2023,GM2018,GMP2018,GM2014,MDW2020,Mogilnykh2022,RSV2018}.
Note that the investigation of regular sets in conjugacy
class schemes is also of recent interest, see~\cite{DMRTP2023,ES2022}.
While many of the association schemes of interest are in fact distance regular graphs,
here we investigate regular sets of lines in polar spaces of rank $3$,
which are in some sense the smallest association schemes associated with finite geometries which are not distance-regular graphs.

The parameters of this association scheme have been long known~\cite{Eisfeld1999,Vanhove2011}.
Special sets of lines have been investigated in the context
of EKR sets of buildings~\cite{IMM2018,Metsch2019}, in the work~\cite{DBDR2023}
on a generalization of ovoids, Boolean degree $1$ functions~\cite{FI2019},
and are also relevant
for the investigation of $1$-systems~\cite{HM1997,ST1994}.
We observe that the line sets of generalized quadrangles,
rank $3$ polar spaces,
spreads and generalized hexagons occur as regular sets
in our association scheme.
We characterize these in several cases.
To our knowledge these characterizations are mostly new.
Lastly, we describe a nontrivial example for a regular
set in $O^+(6, 3)$.

\section{Preliminaries}

There are six families of finite classical rank 3 polar spaces,
the hyperbolic quadrics $O^+(6, q)$,
the parabolic quadrics $O(7, q)$,
the elliptic quadrics $O^-(8, q)$,
the symplectic polar spaces $\Sp(6, q)$,
and the two types of Hermitian polar spaces $U(6,q)$
and $U(7,q)$ (where $q$ is square),
cf.~\cite{BCN,BvM,Vanhove2011}.
Recall that $O(7, q)$ and $\Sp(6, q)$ are isomorphic if $q$ is even.
We will mix vector space dimensions with projective terminology
and call the totally isotropic $1$-spaces \textit{points},
$2$-spaces \textit{lines}, and $3$-spaces \textit{planes}.
Let $\cP$ be a finite classical polar space
with collineation group $G$ and line set $\cL$.
For subspaces $S, T$ of $\cP$, we
denote the points of $T$ which are collinear
with all points of $S$ by $T \cap S^\perp$.
Then the stabilizer of one line $L \in \cL$ has
five orbits
$\cO_{ij}$, $ij \in \{ 00, 10, 11, 20, 21\}$,
on lines which can be described as follows:
\begin{itemize}
 \item $M \in \cO_{00}$ if $L = M$,
 \item $M \in \cO_{10}$ if $\dim(L \cap M)=1$ and $\dim(L \cap M^\perp)=2$,
 \item $M \in \cO_{11}$ if $\dim(L \cap M)=1$ and $\dim(L \cap M^\perp)=1$,
 \item $M \in \cO_{20}$ if $\dim(L \cap M)=0$ and $\dim(L \cap M^\perp)=1$,
 \item $M \in \cO_{21}$ if $\dim(L \cap M)=0$ and $\dim(L \cap M^\perp)=0$.
\end{itemize}
These orbits correspond to the relations of an association scheme.
We write $R_i$ for orbit $\cO_i$, for example $R_{11}$ is the third relation
in the list above.
The eigenvalues for this association scheme can be found in~\cite{Eisfeld1999,Vanhove2011}.
Note that~\cite{Eisfeld1999} has a typo which is corrected in~\cite{Vanhove2011}.
The association scheme has common eigenspaces
$V_{00} = \< \bm{j} \>$, $V_{10}$, $V_{11}$, $V_{20}$, $V_{21}$.
Here $\bm{j}$ denotes the all-ones vector.
For a common eigenspace $V_j$, let $E_j$ be the (idempotent) orthogonal projection
onto it.
Using this ordering of the eigenspaces, the resulting
association scheme has as eigenvalue matrix
\begin{align}
 P =
 \begin{pmatrix}1 & q \left( q{+}1\right) \left( {{q}^{e}}{+}1\right) & {{q}^{e+2}} \left( q{+}1\right)       & {{q}^{e+3}} \left( q{+}1\right) \left( {{q}^{e}}{+}1\right) & {{q}^{2 e+5}} \\
               1 & {{q}^{e+1}}{+}{{q}^{2}}{+}q{-}1                   & q \left( {{q}^{e+1}}{-}1\right)       & q \left( {{q}^{e+2}}{-}{{q}^{e+1}}{-}{{q}^{e}}{-}q\right)   & -{{q}^{e+3}}  \\
               1 & -\left( {{q}^{e}}{+}1\right)                      & {{q}^{e}} \left( {{q}^{2}}{+}1\right) & -{{q}^{e+2}} \left( {{q}^{e}}{+}1\right)                    & {{q}^{2 e+2}} \\
               1 & \left( q{-}1\right)  \left( q{+}1\right)          & -q \left( q{+}1\right)                & -\left( q{-}1\right)  q \left( q{+}1\right)                 & {{q}^{3}}     \\
               1 & -\left( {{q}^{e}}{+}1\right)                      & {{q}^{e}}{-}q                         & q \left( {{q}^{e}}{+}1\right)                               & -{{q}^{e+1}}\end{pmatrix}.\label{Pmat}
\end{align}
Here the parameter $e$ for a rank $d$ polar space is defined such that
$q+1$ is the number of points on a line of the
polar space, and $q^e+1$ is the number of $d$-spaces
through a $(d-1)$-space. Hence, $e \in \{ 0, 1/2, 1, 3/2, 2 \}$.
Note that for a rank 3 polar space,
$n := |\cL| = (q^{e+1}+1)(q^{e+2}+1)(q^2+q+1)$.
The reader can find the dual eigenvalue matrix $Q$,
that is, $Q = nP^{-1}$, as well as a discussion on how to
work with it in~\ref{appendix}.

The inner distribution $a$ of a set $Y$ of lines is the length $5$-vector indexed by
the relations $R_i$, given by
\[
 a_i = |\{ (x, y) \in R_i \,|\, x, y \in Y \}|/|Y|.
\]
We will label the columns of $P$, the rows of the dual eigenvalue matrix $Q$,
and the inner distribution according to the ordering on the relations as above,
e.g.\ the column $20$ of $P$ (the fourth column) contains the eigenvalues of the graph
where two vertices are adjacent if they are in relation $R_{20}$.
Similarly, we will label the rows of $P$ and the columns or $Q$
using our given ordering of the eigenspaces.

For a subset $Y$ of $\cL$, we denote the characteristic vector
of $Y$ by $\chi_Y$.
It is well-known, for instance see Lemma~2.5.1 in~\cite{BCN},
that ${(|Y| \cdot aQ)}_j = n \chi_Y^T E_j \chi_Y$.
Hence, if ${(aQ)}_j = 0$, then $\chi_Y$ is in $V_j^\perp$.
We will need the following standard results.

\begin{Lemma}\label{lem:div_by_example}
      Let $Y, Z \subseteq \cL$.
      Then $|Y \cap Z| = |Y| \cdot |Z|/n$
      if and only if $\chi_Z \in \< \bm{j} \> + \chi_Y^\perp$.
\end{Lemma}
\begin{proof}
      We write $\chi_Y = \frac{|Y|}{n} \bm{j} + \psi_Y$
      and $\chi_Z = \frac{|Z|}{n} \bm{j} + \psi_Z$,
      where $\psi_Y, \psi_Z \in \< \bm{j} \>^\perp$.
      Then $|Y \cap Z| = \chi_Y^T \chi_Z = \frac{|Y| \cdot |Z|}{n}$
      if and only if $\psi_Y^T \psi_Z = 0$.
\end{proof}

Note that the preceeding lemma also holds for weighted sets
with precisely the same proof.
Let $G$ denote the collineation group of $\cP$.
For $Y$ a family of lines of $\cP$, let denote $Y^G$
the orbit of $Y$.
As a special case of Theorem~2.5.16 of~\cite{Vanhove2011} we find

\begin{Lemma}\label{lem:span}
 Let $\chi_Y \in \< \bm{j} \> + \sum_{s \in S} V_s$
      for some $S \subseteq \{ 10, 11, 20, 21 \}$
 such that $\chi_Y \notin \< \bm{j} \> + \sum_{s \in S'} V_s$
 for any proper subset $S'$ of $S$.
 Then $\< \chi_{\tilde{Y}}: \tilde{Y} \in Y^G \> = \< \bm{j} \> + \sum_{s \in S} V_s$.
 In particular, $\chi_Z \in {\left( \sum_{s \in S} V_s \right)}^\perp$
 if and only if $|\tilde{Y} \cap Z| = |\tilde{Y}| \cdot |Z|/n$
 for all $\tilde{Y} \in Y^G$.
\end{Lemma}

We will use Lemma~\ref{lem:span}
repeatedly to characterize sets in certain eigenspaces by their
intersections with other sets. For instance, in \S\ref{sec:genexs}
we will see that the set of lines in a plane lies in
$\< \bm{j} \> + V_{10} + V_{20}$. Hence, we can conclude that
if a set of lines intersects each plane in the same number of
lines, then its characteristic vector is in $\< \bm{j} \> + V_{11} + V_{21}$.\footnote{%
Note that this is all well-known and goes back to Delsarte.}

Since the eigenvalue matrix $P$ of a $d$-class association scheme is invertible,
we have that $Y$ is a regular set if and only if $\chi_{Y} \in \< \bm{j} \> + V_j$
for some eigenspace $V_j$ of the association scheme.
In fact, we can say the following:
\begin{Lemma}\label{lem:reg_set_int_numbers}
 Let $Y$ be a regular set in a $d$-class association scheme
 with $n$ vertices and eigenspaces $V_0 = \< \bm{j}\>$, $V_1$, $\ldots$, $V_d$
 such that $\chi_Y \in \< \bm{j} \> + V_j$.
 Then for any fixed vertex $x$ we have that
 \[
  |\{ y \in Y \mid (x,y) \in R_{i} \}| =
  \begin{cases}
   |Y|\frac{P_{0i}-P_{ji}}{n} + P_{ji} & \mbox{ if $x \in Y$,}    \\
   |Y|\frac{P_{0i}-P_{ji}}{n}          & \mbox{ if $x \notin Y$.}
  \end{cases}
 \]
\end{Lemma}
\begin{proof}
      Write $\chi_Y = \frac{|Y|}{n} \bm{j} + \psi_Y$ with $\bm{j}^T \psi_Y = 0$.
      Then \begin{align*} \chi_{\{x\}}^T A \chi_Y
      &= \chi_{\{x\}}^T \left(P_{0i} \cdot \frac{|Y|}{n} + P_{ji} \cdot \left( \chi_Y - \frac{|Y|}{n} \bm{j} \right) \right)\\
      &= |Y|\frac{P_{0i}-P_{ji}}{n} + P_{ji} \chi_{\{x\}}^T \chi_Y.\qedhere
      \end{align*}
\end{proof}

\section{General Examples}\label{sec:genexs}

Here we list several examples of subsets of lines in rank 3 polar spaces
which lie in a proper subset of the common eigenspaces
of the association scheme.
In many cases, these are either regular sets in all the graphs of the association scheme,
or at least regular with respect to some relation;
in any case, they give us important information about the regular sets through the use of Lemma~\ref{lem:div_by_example}.

\subsection{Combinatorial Designs}\label{subsec:comb_designs}

The group $G$ of a finite classical polar space $\cP$ of rank $d$
acts generously transitively on each set of totally isotropic subspaces of fixed dimension.
Writing $\Omega_{a}$ for the set of isotropic $a$-spaces,
we can consider the $G$-orbits on $\Omega_{a} \times \Omega_{b}$.
These orbits are given precisely by the sets
\[
R_{a,b}^{s,k} = \{ (\pi_{a}, \pi_{b}) \mid \dim(\pi_{a} \cap \pi_{b}) = s,\ \dim(\< \pi_{a}, \pi_{b}\cap\pi_{a}^{\perp}\>) = k\},
\]
with $0 \leq s \leq \min{(a,b)}$, $\max{(a,b)} \leq k \leq \min{(d, a+b-s)}$.
Following~\cite{Ito2004}, we say a set $Y \subset \Omega_{a}$ is
a \textit{combinatorial design with respect to the $b$-spaces of $\cP$}
if, for each such orbit $\cO$, we have that
$|\{\pi_{a} \in Y \mid (\pi_{a}, \pi_{b}) \in \cO\}|$ is constant for all $\pi_{b} \in \Omega_{b}$.

Therefore if $Y$ is a combinatorial design with respect to $b$-spaces of $\cP$,
then every $b$-space is incident with a constant number of elements of $Y$
(we just apply the definition to the relation $R_{a,b}^{a,b}$, respectively $R_{a,b}^{b,a}$, according to whether $a \leq b$ or $a > b$).
This property is in fact enough to characterize combinatorial designs in finite polar spaces;
we give the following theorem specialized only to our context of sets of lines in rank 3 polar spaces to avoid giving needless definitions, but a more general version can be found in~\cite{Vanhove2011}.
Also note that Schmidt's recent research program on combinatorial designs in polar spaces resulted
in a flurry of activities on combinatorial designs in polar spaces~\cite{DIS2024,Kiermaier2024,Weiss2023}.

\begin{Theorem}\label{thm:cdpoints}\label{cor:char_const_point}
  In a rank $3$ polar space with parameter $e$,
  a set $Y$ of lines is a combinatorial design with respect to points
  if and only if every point of $\cP$ lies on precisely $m$ elements of $S$,
  which occurs if and only if $\chi_{Y} \in \< \bm{j} \> + V_{20} + V_{21}$.
  In this case we have that
  \[
  |Y| = m\frac{(q^{e+2}+1)(q^2 +q +1)}{(q+1)}.
  \]
\end{Theorem}

\begin{Theorem}\label{thm:cdplanes}
  In a rank $3$ polar space with parameter $e$,
  a set $Y$ of lines is a combinatorial design with respect to generators
  if and only if every generator of $\cP$ contains precisely $m$ elements of $S$,
  which occurs if and only if $\chi_{Y} \in \< \bm{j} \> + V_{11} + V_{21}$. In this case we have that
  \[
  |Y| = m(q^{e+1} +1)(q^{e+2}+1).
  \]
\end{Theorem}

\subsection{Planes}\label{subsec:planes}

Let $Y$ be the set of lines in a plane.
Then its inner distribution is
\[
 a = (1, q^2+q, 0, 0, 0).
\]
We find that ${(aQ)}_{11} = {(aQ)}_{21} = 0$,
that is $\chi_Y \in \< \bm{j} \> + V_{10} + V_{20}$.
By Lemma~\ref{lem:div_by_example}, we find

\begin{Lemma}\label{lem:plane_div}
 Let $Z$ a family of lines with $\chi_Z \in {(V_{10} + V_{20})}^\perp$.
 Then $|Z| = m(q^{e+1}+1)(q^{e+2}+1)$ for some integer $m$.
\end{Lemma}

%

\subsection{Point-Pencils}\label{subsec:ptpen}

Let $Y$ be a point-pencil, that is
the set of lines through a fixed point $P$.
Then its inner distribution is
\[
 a = (1, q^{e+1}+q, q^{e+2}, 0, 0).
\]
We find that ${(aQ)}_{20} = {(aQ)}_{21} = 0$,
that is $\chi_Y \in \< \bm{j} \> + V_{10} + V_{11}$.
By Lemma~\ref{lem:div_by_example}, we find

\begin{Lemma}\label{lem:pointpencil_div}
 Let $Z$ a family of lines with $\chi_Z \in {(V_{10} + V_{11})}^\perp$.
 Then for some integer $m$, we find
 \begin{enumerate}[(a)]
  \item If $e \neq 1$ and $q$ even, then $|Z| = m(q^2+q+1)(q^{e+2}+1)$.
  \item If $e \neq 1$ and $q$ odd, then $|Z| = m(q^2+q+1)(q^{e+2}+1)/2$.
  \item If $e = 1$, then $|Z| = m(q^4+q^2+1)$.
 \end{enumerate}
\end{Lemma}

%

We can also consider the family $Y'$
of lines in $P^\perp$ which are not incident with $P$.
Then its inner distribution is
\[
 a' = (1, q^2-1, q^{e+1}(q+1), (q^2-1)q^{e+1}, q^{2e+3}).
\]
Again, we find that ${(a'Q)}_{20} = {(a'Q)}_{21} = 0$,
that is $\chi_{Y'} \in \< \bm{j} \> + V_{10} + V_{11}$.

Put $\chi_{\tilde{Y}} =(q^e+1)\chi_Y + \chi_{Y'}$, so that
$\chi_{\tilde{Y}} \in \< \bm{j} \> + V_{10} + V_{11}$.
Then we also have that
\[
 \chi_{\tilde{Y}} = \sum_{\Pi \ni P} \chi_\Pi,
\]
where $\Pi$ is a plane on $P$ and $\chi_\Pi$
is the characteristic vector of all lines in $\Pi$
as in~\S\ref{subsec:planes}, giving us
$\chi_{\tilde{Y}} \in \< \bm{j}\> + V_{10} + V_{20}$.
Hence, we actually have that
$\chi_{\tilde{Y}} \in \< \bm{j}\> + V_{10}$.
This allows us to show the following for line sets in $V_{10}^{\perp}$:

\begin{Lemma}\label{lem:pointpencil_div2}
 Let $Z$ a family of lines with $\chi_Z \in V_{10}^\perp$
 such that at least one point $P$ of the polar space is not covered by $Z$.
 Then for some integer $m$, we find
 \begin{enumerate}[(a)]
  \item If $e \in \{ 0, 2 \}$ and $q$ even, then $|Z| = m(q^{e+2}+1)$.
  \item If $e \in \{ 0, 2 \}$ and $q$ odd, then $|Z| = m(q^{e+2}+1)/2$.
  \item If $e \in \{ \frac{1}{2}, \frac{3}{2}\}$, then $|Z| = m\frac{q^{e+2}+1}{q^e+1}$.
  \item If $e = 1$, then $|Z| = m(q^2-q+1)$.
 \end{enumerate}
\end{Lemma}
\begin{proof}
 Assume that $P$ is as before.
 We find $|\chi_{\tilde{Y}}| = (q^e+1)(q^{e+1}+1)(q^2+q+1)$
 and, as $P$ is not on a line of $Z$, by Lemma~\ref{lem:div_by_example},
 \[
  |\tilde{Y} \cap Z| = |\tilde{Y}| \cdot \frac{(q^e+1)}{q^{e+2}+1}. \qedhere
 \]
\end{proof}


\subsection{1-Systems}\label{subsec:systems}

A $1$-system of a rank $d$ polar space
is a set of lines of size $q^{d+e-1}+1$
which are pairwise opposite. Hence,
in our case it is a set $Y$ of $q^{e+2}+1$ lines
with inner distribution
\[
 a = (1, 0, 0, 0, q^{e+2}).
\]
We find that ${(aQ)}_{10} = 0$, that is $\chi_Y \in V_{10}^\perp$.

Note that $1$-systems are rarely known, cf.~\cite{DBKM2011,HM1997,ST1994}.
One interesting example in our context is a line spread
of $O^-(6, q)$ which is a $1$-system of $O(7, q)$.


\subsection{EKR Sets of Pairwise Opposite Lines}\label{subsec:ekr}

We will see in Proposition~\ref{prop:nothing_in_V21}\footnote{%
While the proof of Proposition~\ref{prop:nothing_in_V21} uses results
of this section, it does not rely on this subsection.}
that the smallest
possible non-empty family $Y$ of lines with $\chi_Y \in \< \bm{j}\> + V_{10}$
has size $(q^{e+1}+1)(q^2+q+1)$ and, by Lemma~\ref{lem:reg_set_int_numbers}, inner distribution
\[
 a = (1, q^{e+1} + q^2+q, q^{e+2}, q^{e+3}, 0).
\]
Hence, as the relation $R_{21}$ does not occur,
such sets $Y$ correspond to EKR sets in buildings as
defined in~\cite{IMM2018}. It is shown in~\cite{Metsch2019}
that a set of pairwise opposite lines
in a rank $3$ building has size at most
$(q^{e+1}+1)(q+1) + q^2(q^e+1)$.
We obtain the following result.

\begin{Lemma}\label{lem:noekrratio}
 There exists no family $Y$ of lines with $\chi_Y \in \< \bm{j} \> + V_{10}$
 and size $(q^{e+1}+1)(q^2+q+1)$.
\end{Lemma}
%

\subsection{Embedded Rank 3 Polar Spaces}\label{subsec:subrank3}

If $\cP \in \{ O(7, q), U(7, q), O^-(8, q) \}$
(or $\cP = \Sp(6, q)$ and $q$ even),
then we find a polar space $\cP'$ of the same rank
with parameter $e-1$ in $\cP$ (namely,
$O^+(6, q), U(6, q)$, or $O(7, q)$). Let $Y$
denote the line set of $\cP'$. Then  its inner distribution is
\[
 a = (1 , q(q+1) ( {{q}^{e-1}}+1)  , {{q}^{e+1}} ( q+1)  , {{q}^{e+2}} ( q+1) ( {{q}^{e-1}}+1)  , {{q}^{2 e+3}}).
\]
We find that ${(aQ)}_{11} = {(aQ)}_{20} = {(aQ)}_{21} = 0$,
so $\chi_Y \in \< \bm{j}\> + V_{10}$.

\subsection{Generalized Quadrangles}\label{subsec:gq}

If $\cP \in \{ O^+(6, q), U(6, q), O(7, q) \}$
(or $\cP = \Sp(6, q)$ and $q$ even),
then we find a polar space $\cP'$ of rank $2$
with parameter $e+1$ in $\cP$ (namely,
$O(5, q)$, $U(5, q)$, or $O^-(6, q)$). Let $Y$
denote the line set of $\cP'$. Then
its inner distribution is
\[
 a = (1, 0, q^{e+1}(q+1), 0, q^{2e+3}).
\]
We find that ${(aQ)}_{10} = {(aQ)}_{20} = {(aQ)}_{21} = 0$,
that is $\chi_Y \in \< \bm{j} \> + V_{11}$.
By Lemma~\ref{lem:div_by_example}, we find
\begin{Lemma}\label{lem:subrank2_div}
 Let $Z$ be a family of lines with $\chi_Z \in V_{11}^\perp$
 for $\cP \in \{ O^+(6, q), \allowbreak U(6, q),\allowbreak O(7, q) \}$.
 Then for some integer $m$, we find
 \begin{enumerate}[(a)]
  \item If $e \in \{ 0, 1, 2 \}$ or if $e = \frac{3}{2}$ and $q \equiv 0,1 \pmod{3}$, then $|Z| = m(q^{e+1}+1)(q^{e+2}+1)$.
  \item If $e = \frac{1}{2}$, then $|Z| = m(q^{\frac{1}{2}}+1)(q^{\frac{5}{2}}+1)$.
  \item If $e = \frac{3}{2}$ and $q \equiv 2 \pmod{3}$, then $|Z| = m(q^{\frac{5}{2}}+1)(q^{\frac{7}{2}}+1)/3$.
 \end{enumerate}
\end{Lemma}

In the upcoming work~\cite{DBDR2023} the authors characterize
nonempty sets of $(d-1)$-spaces of minimal size in $\< \bm{j} \> + V_{11}$
for all rank $d$ polar spaces with $e \in \{ 0, 1/2, 1\}$.
For the sake of self-containment, here we give a short proof
of their result for $d=3$.

\begin{Lemma}
 Let $Y$ be a family of $(q^{e+1}+1)(q^{e+2}+1)$ lines
 such that $\chi_Y \in \< \bm{j}\> + V_{11}$.
 Then $Y$ is a generalized quadrangle of order $(q, q^{e+1})$.
 In particular, such sets do not exist for $e \in \{ 3/2, 2 \}$.
\end{Lemma}
\begin{proof}
 By Lemma~\ref{lem:reg_set_int_numbers}, the inner distribution of $Y$ is $a$ as in the above.
 We will show the standard axioms for finite generalized quadrangles
 of order $(q, q^e)$. Clearly, each line in $Y$ has $q+1$ points.
 Hence, we need to show (1) that each point on one line in $Y$
 lies on $q^{e+1}+1$ lines of $Y$, and (2) that for a line $L \in Y$
 and a point $P$ on a line in $Y$, but not on $L$, there exists
 precisely one line on $P$ which meets $L$ in a point.

 Fix a line $L \in Y$.
 On average a point of $L$ lies on $a_{12}/(q+1) = q^{e+1}$ lines of $Y \setminus \{ L \}$.
 Hence, we find a point $P$ on $L$ which lies on at least $q^{e+1}$ lines of $Y \setminus \{ L \}$.
 In the residue of $P$ we see that $L$ plus these at least $q^{e+1}$ lines of $Y \setminus \{ L \}$ form
 a partial ovoid.
 But a partial ovoid of a generalized quadrangle of order $(q, q^e)$
 has at most size $q^{e+1}+1$,
 so $P$ lies on precisely $q^{e+1}+1$ lines.
 Hence, all points on a line of $Y$ lie on precisely $q^{e+1}+1$ lines of $Y$.

 Now suppose that $M \in Y$ does not contain $P$.
 There cannot be two lines $L_1, L_2 \in Y$
 which contain $P$ and meet $M$ as these are
 necessarily in relation $R_{11}$ (which is forbidden).

 It remains to show that there is one line $N \in Y$
 through $P$ which meets $M$.
 As there are no two lines through $P$ meeting $M$,
 the number of lines of $Y$ not through $P$ meeting a line
 through $P$ equals $(q^{e+1}+1) \cdot q \cdot q^{e+1}$.
 This equals the number of all lines of $Y$ not through $P$.
\end{proof}

\subsection{Ovoids and Point-Pencils}\label{subsec:ovoids}

Suppose that $\cP$ contains an ovoid $\cO$,
that is a set of $q^{e+2}+1$ pairwise non-collinear
points. A notable property of an ovoid is that each
plane of $\cP$ meets $\cO$ in precisely one point.
Take $Y$ to be the union of the point-pencils (of lines)
through the elements of $\cO$. We saw in \S\ref{subsec:ptpen}
that a point-pencil lies in $\<\bm{j}\> + V_{10} + V_{11}$,
so also $\chi_Y \in \<\bm{j}\> + V_{10} + V_{11}$.
The Example in \S\ref{subsec:planes} shows that $\chi_Y \in \<\bm{j}\> + V_{11} + V_{21}$.
Hence, $\chi_Y \in \<\bm{j}\> + V_{11}$.
Note that $|Y| = (q+1)(q^{e+1}+1)(q^{e+2}+1)$.

Ovoids in rank 3 polar spaces are somewhat rare,
in fact it has been shown in~\cite{T1981} that
$\Sp(6,q)$, $O(7,q)$ for $q$ even, $O^-(8,q)$, and $U(7,q)$
contain no ovoids,
in~\cite{DBM} that $U(6,4)$ contains no ovoid,
and in~\cite{BGS} that $O(7,q)$ contains no ovoid
when $q > 3$ is prime.
Ovoids are known to exist in $O^+(6,q)$ for all $q$
(here they are images under the Klein correspondence
of spreads in $\PG(3,q)$),
and in $O(7,q)$ when $q = 3^{h}$~\cite{BKLP2007}.

\subsection{Constructions from \texorpdfstring{$m$}{m}-ovoids}\label{subsec:movoid}

Take $\cP$ and $\cP'$ as in \S\ref{subsec:gq}, that is,
$\cP \in \{ O^+(6, q), U(6, q), O(7, q) \}$
(or $\cP = \Sp(6, q)$ with $q$ even),
and $\cP$ is an embedded
$O(5, q)$, $U(5, q)$, or $O^-(6, q)$, respectively.
An \textit{$m$-ovoid} of $\cP'$ is a set
of points $\cO$ of $\cP'$ such that each line of $\cP'$
meets $\cO$ in $m$ points.
It follows that $|\cO| = m(q^{e+2}+1)$.
The notion of an $m$-ovoid was introduced by Thas~\cite{T1989};
the special case of a hemisystem, that is an $m$-ovoid with $m = (q+1)/2$,
even goes back to Segre~\cite{Segre1965}.

Suppose we have an $m$-ovoid $\cO$ of $\cP' \leq \cP$.
Let $Y$ be the set
of lines in $\cP$ which meet $\cP$ precisely in a
point of $\cO$, and let $Z$ be the set of lines of $\cP'$.
For a point $P \in \cP$, let $Y_P$ be the point-pencil of lines through $P$.
Then
\[
 \chi_{Y} = \left( \sum_{P \in \cO} \chi_{Y_P} \right) - m\chi_{Z}.
\]
As $\chi_{Y_P} \in \< \bm{j} \> + V_{10} + V_{11}$ (see \S\ref{subsec:ptpen}) and
$\chi_{Z} \in \< \bm{j} \> + V_{11}$, we find that $\chi_Y \in \< \bm{j} \> + V_{10} + V_{11}$.
Now each plane of $\cP$ meets $\cP'$ in a line, and so
contains precisely $mq$ lines of $Y$,
hence $\chi_Y \in {(V_{10}+V_{20})}^\perp$ (see \S\ref{subsec:planes}).
This allows us to conclude the following:

\begin{Lemma}\label{lem:movoid}
 Let $q$ be odd.
 If $O(5, q)$, $U(5, q)$, respectively, $O^-(6, q)$
 contains an $m$-ovoid, then
 $O^+(6, q)$, $U(6, q)$, respectively, $O(7, q)$
 contains a family $Y$ such that $\chi_{Y} \in \< \bm{j}\> + V_{11}$
 and
 \[
  |Y| = mq(q^{e+1}+1)(q^{e+2}+1).
 \]
 In particular, $Y$ exists for $O^+(6, q)$ and $O(7, q)$.
\end{Lemma}
\begin{proof}
 We already saw $\chi_Y \in \< \bm{j}\> + V_{11}$.
 Then number of lines through a point $P$ of $\cP'$
 which are not in $\cP'$ is $q(q^{e+1}+1)$. Hence,
 the expression for $|Y|$ follows from $|Y| = q(q^{e+1}+1)|Y'|$.
\end{proof}

There are several known infinite families and sporadic examples of $m$-ovoids that can be used in this construction for
$O^-(6,q) \leq O(7,q)$~\cite{BLMX2018,CP2017,CP2005,LV2023},
and for $O(5,q) \leq O^{+}(6,q)$~\cite{BLP2009, CCEM, FMX2016, FT}.
There are currently no known examples of $m$-ovoids in $U(6,q)$.

Note that Lemma~\ref{lem:movoid} was shown in~\cite{DBDR2023}
for the special case of hemisystems in $O^-(6,q)$,
although our proof here is different.

\subsection{Boolean Degree 1 Functions}

In~\cite{FI2019} a family of lines $Y$ with $\chi_Y \in \< \bm{j}\> + V_{10} + V_{11}$
is called a Boolean degree $1$ function.\footnote{%
As these are precisely the family of lines in the span of the characteristic
vectors of point-pencils, to if we consider points as variables, then
vectors in $\< \bm{j}\> + V_{10} + V_{11}$ are the natural analog of
affine functions on the hypercube.} There it is
conjectured that all such sets are trivial in a certain sense
if the rank of the polar space is large enough.
More formally, we saw that we find these examples for $Y$ with
$\chi_Y \in \< \bm{j}\> + V_{10} + V_{11}$:
\begin{enumerate}[(A)]
 \item The set of all lines through a point, see \S\ref{subsec:ptpen}.
 \item The set of all lines in a degenerate hyperplane, see \S\ref{subsec:ptpen}.
 \item The set of all lines in a nondegenerate hyperplane, see \S\ref{subsec:gq}.
\end{enumerate}
We also know that if $Y$ and $Z$ satisfy $\chi_Y, \chi_Z \in \< \bm{j}\> + W$,
then also
\begin{enumerate}[(I)]
 \item the complement $\overline{Y}$ of $Y$ satisfies $\chi_{\overline{Y}}  \in \< \bm{j}\> + W$,
 \item if $Y \subset Z$, then $\chi_{Y \setminus Z}  \in \< \bm{j}\> + W$,
 \item if $Y \cap Z = \emptyset$, then $\chi_{Y \cup Z}  \in \< \bm{j}\> + W$.
\end{enumerate}

Motivated by results on the hypercube, the Johnson graph,
the conjugacy class scheme of the symmetric group and other association schemes,
Conjecture 5.1 in~\cite{FI2019} was given. Limited to our case of lines, it can be paraphrased as follows:
\begin{Conjecture}\label{conj:deg1}
      There exists a constant $d_0(q)$ such that if $\cP$ is a rank $d$ polar space
      with $d \geq d_0(q)$ and $Y$ is a family of lines with $\chi_Y \in \< \bm{j}\> + V_{10} + V_{11}$,
      then $Y$ can be obtained
by combining the examples (A)--(C) with the operations (I)--(III) (with $W = V_{10} + V_{11}$).
\end{Conjecture}

In~\cite{FI2019}, this is formally shown for lines in $O^+(2d, 2)$ for $d\geq 3$.
Hence, for our case of rank $3$ polar spaces Conjecture \ref{conj:deg1} holds for this case.
Note that Theorem 5.2 in~\cite{FI2019} is not stated correctly for the case $d=3$
(while being correct for the other cases) as it claims that at most
one non-degenerate hyperplane section can occur. The actually correct statement
can be found in Conjecture 5.1~in \cite{FI2019}.

The analogous version of Conjecture \ref{conj:deg1} has been recently shown
for vector spaces, see~\cite{Ihringer2023}. Maybe this adds some plausibility
to Conjecture \ref{conj:deg1}.
The examples derived above from $m$-ovoids show that
the assertion of the conjecture does not hold
for $O^+(6, q)$ or $O(7, q)$ and $q$ odd.
Generalizing the $m$-ovoid construction given here to general
rank would most likely disprove the conjecture.

\subsection{Spreads}\label{subsec:spreads}

A spread of a polar space of rank $d$
is a set of $q^{d+e-1}+1$ pairwise disjoint
maximal subspaces of $\cP$.
Suppose that the rank 3 polar space $\cP$ contains a spread $\cS$,
that is a set of $q^{e+2}+1$ pairwise disjoint planes.
Spreads always exist for $\Sp(6, q)$,
never exist for $O^+(6, q)$ and $U(6, q)$,
and the situation is unclear for other rank 3 polar spaces, cf.~\cite{DBKM2011}.
Let $Y$ denote the set of lines which
are contained in one element of $\cS$.
Then the inner distribution of $Y$ is
\[
 a = (1, q^2+q, 0, q^{e+2}(q+1), q^{e+4}).
\]
We find that ${(aQ)}_{10} = {(aQ)}_{11} = {(aQ)}_{21} = 0$,
that is $\chi_Y \in \< \bm{j}\> + V_{20}$.
By Lemma~\ref{lem:div_by_example}, we find
\begin{Lemma}
 Let $Z$ be a family of lines in a rank 3 polar space
 with a spread and $\chi_Z \in V_{20}^\perp$. Then
 $|Z| = m(q^{e+1}+1)$ for some integer $m$.
\end{Lemma}

\subsection{Generalized Hexagons}\label{subsec:gh}

If $\cP = O(7, q)$, then $\cP$ contains
the split Cayley hexagon, see~\cite{Tits1959}. Its line
set $Y$ has the inner distribution
\[
 a = (1, q^2+q, 0, q^4+q^3, q^5).
\]
We find that ${(aQ)}_{10} = {(aQ)}_{11} = {(aQ)}_{21} = 0$,
that is $\chi_Y \in \< \bm{j}\> + V_{20}$.

Now that we know two examples,
we can characterize families $Y$ with $\chi_Y \in \< \bm{j} \> + V_{20}$
of minimal size\footnote{We will give an argument in Proposition~\ref{prop:nothing_in_V21}.}, that is $|Y| = (q^{e+2}+1)(q^2+q+1)$
and $Y$ has inner distribution
\[
 a = (1, q^2+q, 0, q^{e+2}(q+1), q^{e+4}).
\]
From now on we assume that $Y$ is as above.
Theorem~\ref{cor:char_const_point} shows the following lemma.
\begin{Lemma}\label{lem:hex_pt}
      Each point of $\cP$ lies on $q+1$ lines of $Y$.
\end{Lemma}

\begin{Lemma}\label{lem:V02_ln_pl_int}
 For any plane $\Pi$ of $\cP$, we have $|Y \cap \Pi| \in \{ 0, 1, q+1, q^2+q+1 \}$,
 and if $|Y \cap \Pi| = q+1$, then the lines of $Y \cap \Pi$ are a point-pencil,
 that is the $q+1$ lines meet in a common point.
 Here $|Y \cap \Pi| = q+1$ cannot occur if $e \in \{ 0, 1/2 \}$.
\end{Lemma}
\begin{proof}
 As the relation $R_{21}$ does not occur,
 all $q+1$ lines in $Y$ through a point
 are coplanar. Hence, as $a_{11} = q^2+q$,
 each point on one line of $Y$ is on
 precisely $q+1$ lines of $Y$ (and these
 lines are coplanar). Hence, if a plane $\Pi$
 contains at least $2$ lines $L_1, L_2$ of $Y$,
 then $Y$ contains at least the $q+1$
 lines of $\Pi$ through $L_1 \cap L_2$.
 Furthermore, if a plane $\Pi$ contains
 at least $q+2$ lines, then $Y$ contains all
 lines of $\Pi$.

 \smallskip

 For $e \in \{ 0, 1/2 \}$ there are less than
 $q+1$ planes through a line, so a plane
 must contain $0$, $1$, or $q^2+q+1$ lines of $Y$.
\end{proof}

Call $\Pi$ a $p$-plane if $|Y \cap \Pi| = p$.

\begin{Lemma}\label{lem:V02_pl_meets}
 Let $\Pi$ be a $p$-plane and $\Sigma$ be an $s$-plane.
 \begin{enumerate}[(a)]
  \item If $p = q^2+q+1$ and $\Pi \cap \Sigma$ is a line,
        then $s=1$.
  \item If $p = q^2+q+1$ and $\Pi \cap \Sigma$ is a point,
        then $s \in \{ 0, 1 \}$.
  \item If $p = q+1$ and $\Pi \cap \Sigma$ is a line of $Y$,
        then $s \in \{ 1, q+1 \}$. If $e=1$, then $s = q+1$.
  \item If $p = q+1$ and $\Pi \cap \Sigma$ is a line not in $Y$,
        then $s = 0$.
 \end{enumerate}
\end{Lemma}
\begin{proof}
 For (a): Put $L = \Pi \cap \Sigma$.
 Clearly, $\Sigma$ contains at least $L \in Y$.
 If $\Sigma$ contains some other line $M \in Y$,
 then the point $L \cap M$ lies on $q+2$ lines of $Y$
 which contradicts Lemma~\ref{lem:hex_pt}. Hence, $\Sigma$ contains
 precisely one line of $Y$.
 This also implies (c) except for $e=1$.
 For $e=1$, $s=1$ cannot occur as there are only $q+1$
 planes through a line.

 For (b): Put $R = \Pi \cap \Sigma$.
 Suppose that $\Sigma$ contains at least two lines $M, M' \in Y$.
 By Lemma~\ref{lem:V02_ln_pl_int}, there exists a line $N \in Y \cap \Sigma$
 such that $N$ contains $R$. Then $N$ and $q$ lines of $\Pi$
 are in the forbidden relation $R_{21}$.

 For (d): Put $S = \Pi \cap \Sigma$.
 Suppose that $\Sigma$ contains a line $M$ of $Y$.
 Then $M$ is not in $\Pi$, so $M$ meets $\Sigma$
 in a point $R$. By Lemma~\ref{lem:V02_ln_pl_int},
 one line $N$ of $Y \cap \Pi$
 contains $R$. Then $M$ and $N$ are in a forbidden relation $R_{21}$
 which is a contradiction.
\end{proof}

\noindent
In particular, lines in $(q^2+q+1)$-planes never meet
lines in $(q+1)$-planes. We say that a set of lines $Z$
is a geometry of order $(s, t)$ if each line of $Z$ contains
$s+1$ points and each point lies on $t+1$ lines of $Z$.
In~\cite{TVM1998}, an embedding of a geometry $Z$ of order $(q,t)$
in a projective space of order $q$ is called \textit{full},
and it is called \textit{flat} if the lines of $Z$
through a point of $Z$ are coplanar.
In~\cite{TVM1998}, full and flat embeddings of the split Cayley
hexagon are characterized.

\begin{Proposition}\label{prop:class_hex_spread}
 Let $Y$ be a family of lines of size
 $(q^{e+2}+1)(q^2+q+1)$ such that $\chi_Y \in \< \bm{j} \> + V_{20}$.
 Then $Y$ can be partitioned into families of $Y_1$ and $Y_2$ such that
 the following holds.
 \begin{enumerate}[(a)]
  \item All lines in $Y_1$ are disjoint from all lines in $Y_2$.
  \item All planes intersect $Y_1$ in $1$ or $q^2+q+1$ lines.
        The planes with $q^2+q+1$ lines of $Y_1$ form a partial spread of $\cP$.
  \item All planes intersect $Y_2$ in $0, 1$ or $q+1$ lines.
        Then $Y_2$ is a geometry of order $(q,q)$ with $|Y_2| \geq (q^3+1)(q^2+q+1)$
        which does not contain triangles, quadrangles, or pentagons.
        In case of equality, $\cP = O(7, q)$ (respectively, $\cP = \Sp(6, q)$, $q$ even) and $Y_2$ is
        the natural embedding of the split Cayley hexagon in $O(7, q)$ (respectively, in $\Sp(6, q)$).
 \end{enumerate}
 In particular, if $e \in \{ 0, 1/2 \}$, then $Y$ does not exist,
 and if $\cP = \Sp(6, q)$ and $q$ odd, then $Y_2$ is empty.
\end{Proposition}
\begin{proof}
 Let $L \in Y_1$ if $L \in Y$ and $L$ lies in a $(q^2+q+1)$-plane.
 Let $L \in Y_2$ if $L \in Y$ and $L$ lies in a $(q+1)$-plane.
 Lemma~\ref{lem:V02_pl_meets} shows that this is a partition of $Y$,
 that (a) holds, and that the $(q^2+q+1)$-planes form a partial spread $\cS$.

 \smallskip

 By Lemma~\ref{lem:V02_pl_meets}, if $e \in \{ 0, 1/2\}$, then $Y_2$
 is empty. Hence, $\cS$ is a spread, but it is well-known
 that $O^+(6, q)$ and $U(6, q)$ do not possess spreads, cf.~\cite{DBKM2011}.
 Hence, $Y$ does not exist.

 \medskip

 For (c), we already saw that each point on a line of $Y_2$
 lies in precisely $q+1$ lines of $Y_2$ in Lemma~\ref{lem:hex_pt}.
 Hence, $Y_2$ is the natural embedding of a point-line geometry
 of order $(q,q)$ in $\cP$
 (and $\cP$ is then naturally embedded in projective geometry
 of order $q$ of dimension $5$, $6$, or $7$,
 or, in the language used here, a vector space of dimension
 $6$, $7$, or $8$ over the finite field with $q$ elements).
 Also note that the embedding is flat.

 \smallskip

 We want to show that $Y_2$ does not contain triangles,
 quadrangles or pentagons. For this we will repeatedly
 use that $\< L, M\>$ is a plane of $\cP$ if $L, M \in Y_2$ meet.

 There is no triangle in $Y_2$: Suppose that $L_1, L_2, L_3 \in Y_2$
 is a triangle. Then $\< L_1, L_2, L_3 \>$ is a plane $\Pi$ of $\cP$.
 But then $\Pi$ is a $(q^2+q+1)$-plane.

 There is no quadrangle in $Y_2$: Suppose that $L_1, L_2, L_3, L_4 \in Y_2$
 is a quadrangle. Then the plane $\Pi := \<L_1, L_2\>$ meets the plane $\< L_3, L_4\>$
 in the line $\< L_1 \cap L_4, L_2 \cap L_3\>$. But this contradicts
 Lemma~\ref{lem:V02_pl_meets}.

 There is no pentagon in $Y_2$: Suppose that $L_1, L_2, L_3, L_4, L_5 \in Y_2$
 is a pentagon.
 Then $L_1^\perp$ contains $P_1 := L_2 \cap L_3 \subseteq \< L_1, L_2 \>$
 and $P_2 := L_4 \cap L_5 \subseteq \< L_5, L_1\>$.
 But then $L_1^\perp$ contains the line $\< P_1, P_2 \>$ of the $(q+1)$-plane $\Sigma := \< L_3, L_4\>$.
 Hence, $L_1$ contains a point $Q$ of $\Sigma$.
 Then, assuming that
 $Q \neq P_1$, $L_1, L_2, L_3, \< L_3 \cap L_4, Q \>$ is a quadrangle.

 Counting the number of lines of $Y_2$ by their distance to
 a fixed point on some line of $Y_2$ shows that $|Y_2| \geq q+1 + q^2(q+1) + q^4(q+1) = (q^3+1)(q^2+q+1)$.
 Then the main result of~\cite{TVM1998} shows that $Y_2$ is the line set
 of the natural embedding of the Split Cayley hexagon in $\Sp(6, q)$, $q$ even,
 or in $O(7, q)$.

 We still need to rule out $e \in \{ 3/2, 2 \}$ in case of equality.
 If $\cP = U(7, q)$, then there exists no
 natural embedding of $O(7, q)$ in $U(7, q)$.
 If $\cP = O^-(8, q)$, then any plane will
 meet a naturally embedded $O(7, q)$ in a line, so $Y_1$ must be empty.
 Hence, $|Y_2| > (q^3+1)(q^2+q+1)$.
\end{proof}

Note that the main results in either~\cite{Ihringer2014} or~\cite{TVM2008}
characterize $Y_2$ as well except for $\cP = \Sp(6, q)$. We are unaware
of a result which covers the symplectic case.

\begin{Question}
 If $e \in \{ 3/2, 2\}$, can $Y_2$ be nonempty?
\end{Question}

\section{Divisibility Conditions}

Lemma~\ref{lem:reg_set_int_numbers} implies the following,
which is very helpful for divisibility conditions of regular sets.

\begin{Lemma}\label{lem:gen_div_cond_P}
 Let $Y$ be a regular set in a $d$-class association scheme
 with $n$ vertices and eigenspaces $V_0 = \< \bm{j}\>$, $V_1$, $\ldots$, $V_d$
 such that $\chi_Y \in \< \bm{j} \> + V_j$.
 Then $|Y| \cdot \frac{P_{0i}-P_{ji}}{n}$ is an integer for each relation $R_{i}$.
\end{Lemma}

Let us summarize all divisibility conditions.

\begin{Proposition}\label{prop:nothing_in_V21}
 Let $Y$ be a family of vertices in an association scheme
 with $P$-matrix~\eqref{Pmat}, where $e \in \{ 0, 1/2, 1, 3/2, 2 \}$,
 such that $\chi_Y \in \< \bm{j} \> + V_{ij}$. Then there exists an
 integer $m$ such that the following holds:
 \begin{enumerate}[(a)]
  \item If $ij = 10$, then $|Y| = m(q^{e+1}+1)(q^2+q+1)$ with $m \neq 1, q^{e+2}$.
  \item If $ij = 11$, then $|Y| = m(q^{e+1}+1)(q^{e+2}+1)$.
  \item If $ij = 20$, then:
  \begin{enumerate}[(i)]
   \item If $e\neq 1$ and $q$ even, then $|Y| = m(q^2+q+1)(q^{e+2}+1)$.
   \item If $e \neq 1$ and $q$ odd, then $|Y| = m(q^2+q+1)(q^{e+2}+1)/2$ with $m \neq 1, 2q^{e+2}+1$.
   \item If $e=1$, then $|Y| = m(q^4+q^2+1)$ with $m \in \{0\} \cup [q+1, q^2(q+1)] \cup \{ (q^2+1)(q+1) \}$.
  \end{enumerate}
  \item If $ij = 21$, then $|Y| = 0$ or $|Y| = n$.
 \end{enumerate}
\end{Proposition}
\begin{proof}
 We use Lemma~\ref{lem:gen_div_cond_P} for the case $ij=10$ with
 \[
  \frac{P_{00,10} - P_{10,10}}{n}.
 \]
 We find that $|Y| \cdot \frac{1}{(q^2+q+1)(q^{e+1}+1)}$ is an integer.
 Lemma~\ref{lem:noekrratio} rules out $m=1$.
 The case $m=q^{2e}$ follows by considering the complement of $Y$.

 Lemma~\ref{lem:plane_div} shows the case $ij=11$.

 Lemma~\ref{lem:pointpencil_div} shows the divisibility condition in case $ij=20$.
 For the additional conditions on $m$, by Lemma~\ref{lem:reg_set_int_numbers}
 the number of elements of $Y$ in relation $R_{11}$ to some fixed element of $Y$ is
 \[
  0 \leq |Y| \frac{P_{00,11}}{n} + P_{20,11} =
  \begin{cases}
      m \cdot \frac{q^2+q}{2} - q^2-q & \text{ if } e\neq 1 \text{ and } q \text{ odd,}\\
      mq - q^2-q & \text{ if } e=1.
  \end{cases}
 \]
 This excludes $m=1$ for $e \neq 1$ and $q$ odd, respectively, $m \in [1, q]$ for $e=1$.
 The remaining cases follow by considering the complement of $Y$.

 Lemma~\ref{lem:plane_div} and Lemma~\ref{lem:pointpencil_div} combined
 show the case $ij=21$.
%
%
\end{proof}

\section{Extremal Sets with Forbidden Relations}

The famous ratio bound (also called Delsarte-Hoffman bound) states that
a coclique $Y$ in a regular graph on $n$ vertices, degree $k$,
and smallest eigenvalue $\lambda$ has at most size $n/(1-k/\lambda)$.
In our case, we can apply it to cocliques in any graph defined by a relation
(or union of relations) in our association scheme.
If $|Y| = n/(1-k/\lambda)$, then
$\chi_Y \in \< \bm{j}\> + W$, where $W$ is the eigenspace of $\lambda$ (for instance see Proposition 3.7.2 in~\cite{BCN}).
Now consider the graph $\Gamma$ with adjacency defined as the union
of relations $R_i$ with $i \in \cI$ for some subset of relations $\cI$.
Then, more generally, Delsarte's linear programming bound
(see Proposition 2.5.2 in~\cite{BCN})
states that if $Y$ has inner distribution $a$, then
\[
  {(aQ)}_j \geq 0
\]
for all $j$, where $j$ runs over the indices of the eigenspaces $E_j$.
If we set $a_i = 0$ for all $i \in \cI$
and we maximize for $\sum a_i$,
then this gives us a bound on the size of a coclique in $\Gamma$.
Schrijver showed that a weighted version of the Delsarte-Hoffman
bound, also known as Lov\'asz bound, is equivalent to Delsarte's LP bound~\cite{Schrijver1979}.
Let $\cJ$ be the set of indices of those eigenspaces $V_j$ for which
${(aQ)}_j = 0$ for all optimal solution of the linear program stated in the above.
Then an optimal solution satisfies $\chi_Y \in \< \bm{j} \> + W$, where $W = \bigcup_{j \not\in \cJ} V_j$.

Let us briefly discuss the possible families of lines for which some
subset of relations is forbidden. Also note that we have a constant
number of variables, a constant number of inequalities,
and that the solution of the linear program is necessarily a polynomial in $q$.
Hence, determining the optimal solution for each case is a (small)
finite problem and can done relatively quickly by computer.


\bigskip
\noindent
Relation $R_{10}$ forbidden: For a family of lines $Y$ in $\cP$ with no two in a plane,
the ratio bound shows $|Y| \leq (q^{e+1}+1)(q^{e+2}+1)$ with equality if
$\chi_Y \in \< \bm{j}\> + V_{11} + V_{21}$. For $e \in \{ 0, 1/2, 1\}$,
generalized quadrangles are examples, see \S\ref{subsec:gq}.\
Also disallowing $R_{20}$ in addition to $R_{10}$ does not change the bound.

\medskip
\noindent
Relation $R_{11}$ forbidden: For a family of lines $Y$ in $\cP$ with any two either disjoint
or in a common plane, the ratio bound shows $|Y| \leq (q^{e+2}+1)(q^2+q+1)$
with equality if $\chi_Y \in \< \bm{j}\> + V_{20}$, see~\S\ref{subsec:spreads}
and~\S\ref{subsec:gh}.

\medskip
\noindent
Relation $R_{20}$ forbidden: For a family of lines $Y$ in $\cP$ with any two either meeting in
a point or opposite, if $e \geq 1/2$ or $(e,q) = (0,2)$,
then the ratio bound shows $|Y| \leq (q^{e+1}+1)(q^{e+2}+1)$
with equality if $\chi_Y \in \< \bm{j}\> + V_{11}$. We saw in~\S\ref{subsec:gq},
that equality is obtained if and only if $\cP = O^+(6, 2)$,
$\cP = U(6, q)$, $\cP = O(7, q)$, or $\cP = \Sp(6, q)$ with $q$ even.
For $e=0$ and $q \geq 3$, Delsarte's LP bound is better and yields
precisely the same bound as when $R_{11}$ and $R_{20}$ are forbidden,
see below.

\medskip
\noindent
Relation $R_{21}$ forbidden: The largest families of lines $Y$ in $\cP$ with no two opposite were discussed in~\S\ref{subsec:ekr}.


\medskip
\noindent
Relations $R_{10}, R_{11}$ forbidden: For a family of lines $Y$
without intersecting lines, Delsarte's LP bound yields $|Y| \leq \frac{(q^2+q+1)(q^{e+2}+1)}{q+1}$
with equality only if $\chi_Y \in \< \bm{j}\> + V_{20} + V_{21}$.
Equality occurs only if $Y$ is a \textit{line spread} of $\cP$.
Clearly this can only happen if $e=1$; these are known to exist for small $q$,
but in general existence seems open, see Remark 8 in~\cite{MDW2020}.


\medskip
\noindent
Relations $R_{10}, R_{21}$ forbidden: Consider a family of lines $Y$
with no two lines in a plane and no two lines opposite.
If $(e, q) = (0, 2)$ or $e > 0$,
then Delsarte's LP bound yields
$|Y| \leq \frac{((q^2+q+1)(q^{e+1}+1)(q^{e+1}+2q-1)}{q^{e+2}+q^3+q-1}$
and is never an integer.
If $e = 0$ and $q >2$, then Delsarte's LP bound yields
$|Y| \leq \frac{2q^2+q-1}{q-1}$ and is only an integer for $q=3$.
Equality occurs only if $\chi_Y \in V_{20}^\perp$.

\medskip
\noindent
Relations $R_{11}, R_{20}$ forbidden: For a family of lines $Y$ in $\cP$ with any two either in a common plane
or opposite, we have that Delsarte's LP bound yields $|Y| \leq \allowbreak \frac{(q^{e+2}+1)(q^{e+2}+q^3+q-1)}{q^{e+1}+2q-1}$
with equality only if $\chi_Y \in \< \bm{j}\> + V_{11} + V_{20}$.
The upper bound is only an integer for $(e, q) \in \{ (0, 2), (0, 5), (0, 7), (1, 3) \}$.

\medskip
\noindent
Relations $R_{11}, R_{21}$ forbidden: Delsarte's LP bound yields $|Y| \leq q^2+q+1$
with equality only if $\chi_Y \in V_{11}^\perp$. See \S\ref{subsec:planes}.
Disallowing one more relation does not change the bound.

\medskip
\noindent
Relations $R_{20}, R_{21}$ forbidden: This corresponds to EKR sets of pairwise intersecting lines
and has been solved in \cite{Metsch2016}.


\medskip
\noindent
Relations $R_{10}, R_{11}, R_{20}$ forbidden: This corresponds to \S\ref{subsec:systems}.


\medskip
\noindent
Relations $R_{10}, R_{20}, R_{21}$ forbidden: Delsarte's LP bound shows $|Y| \leq q^{e+1}+1$
with equality only if $\chi_Y \in V_{20}^\perp$. This bound is tight as an ovoid
in the quotient of a point $P$ has size $q^{e+1}+1$.


\begin{Problem}
 Determine the largest families in all of the cases above.
\end{Problem}

\section{Two-weight Codes \& SRGs}

It was noted in~\cite{ST1994} (among other things) that for
$\cP \in \{ \Sp(2d, q), U(2d+1, q), O^-(2d+2, q) \}$ naturally embedded
in an ambient ($\Ff_q$) vector space $V$,
the collection of points lying on an element points of an $1$-system give
a two-weight code, that is, the hyperplanes of $V$ have two intersection sizes
with respect to this set of points.
Similar observations can be found in~\cite{BKLP2007} for the strongly
regular graph of the points of $\cP$.
Here we discuss these results more generally, in the context of rank 3 polar spaces.
Note that in the general result $m$ is not necessarily an
integer for $e \neq 2$, but we chose this parameterization for the
sake of readability.

\begin{Theorem}\label{thm:twoweight}
 Let $Y$ be a family of pairwise non-intersecting lines in
 $\cP \in \{ \Sp(6, q), U(7, q), O^-(8, q) \}$
 with $\chi_Y \in V_{10}^\perp$.
 Write $|Y| = m(q^{e+2}+1)$ and let $H$ be a hyperplane of $V$.
 If $H$ is a nondegenerate hyperplane or a degenerate hyperplane
 with its radical not in $Y$, then $H$ contains precisely
 \[
  m(q+1)(q^{e+1}+1)
 \]
 points which lie on a line of $Y$, while if $H$ is a degenerate hyperplane with its radical in $Y$,
 then $H$ contains precisely
 \[
  m(q+1)(q^{e+1}+1) - q^{e+1}
 \]
 points which lie on a line of $Y$.
\end{Theorem}
\begin{proof}
 A hyperplane $H$ of $V$ intersects $\cP$ in a degenerate polar space
 or in a rank $3$ polar space with parameter $e-1$.
 Let $Z$ denote the lines of $H$ in $\cP$.
 We will use Lemma~\ref{lem:div_by_example}.

 We saw in~\S\ref{subsec:subrank3} that if $H$ is a rank $3$ polar space
 (that is, $H$ is nondegenerate hyperplane),
 then $\chi_Z \in \< \bm{j}\> + V_{10}$. We find
 \[
  |Z| = (q^{e}+1)(q^{e+1}+1)(q^2+q+1).
 \]
 Hence, $H$ contains $m(q^e+1)$ lines of $Y$
 while it intersects all the other lines of $Y$ in a point.

 Now let $H$ be a degenerate hyperplane with radical $H^\perp$.
 Similarly,
 we saw for the $\tilde{Y}$ in \S\ref{subsec:ptpen} that
 $\chi_{\tilde{Y}} \in \< \bm{j}\> + V_{10}$ and
 $|\chi_{\tilde{Y}}| = (q^e+1)(q^{e+1}+1)(q^2+q+1) = |Z|$.
 Hence, with $Y'$ as in \S\ref{subsec:ptpen} denoting the set
 of lines through $H^\perp$,
 \[
  \chi_Z^T \chi_{\tilde{Y}} = | Z \cap \tilde{Y} | + q^e | Z \cap Y' |.
 \]
 This shows the claims for $H$ being degenerate.
\end{proof}


Recall that a \textit{strongly regular graph} is a $k$-regular
graph of order $v$ (not complete, not edgeless) with two non-principal eigenvalues $r, s$.
Following~\cite{BvM}, we write $k \geq r \geq 0 > s$.
A classical result due to Delsarte, using the standard terminology
for strongly regular graphs of~\cite[p.~175]{BvM}, shows the following.

\begin{Corollary}\label{cor:srg}
 Let $Y$ be as in Theorem~\ref{thm:twoweight}
 with $\cP \in \{ \Sp(6, q), U(7, q), O^-(8, \allowbreak q) \}$.
 Then there exists a strongly regular graph with
 parameters
 \begin{align*}
   & v = q^{4+2e},            &  & r = m(q^2-1),    \\
   & k = m(q^{e+2}+1)(q^2-1), &  & s = r - q^{e+2}.
 \end{align*}
\end{Corollary}

We will assume that $\chi_Y \in V_{10}^\perp$
and that the inner distribution of $Y$ has the form
\[
 a = (1, 0, 0, x, y).
\]

Let us discuss all the cases with at most 1300 vertices.

\bigskip

First we discuss $\cP = \Sp(6, q)$.
We obtain a strongly regular graph with parameters
\begin{align*}
  & v = q^{6},             &  & r = m(q^2-1),        \\
  & k = m(q^{3}+1)(q^2-1), &  & s = -m(q^{3}-q^2+1).
\end{align*}
For $m=1$, we have a $1$-system
and can take a spread of $O^-(6, q)$ for this.
We obtain the affine polar graph $VO^-(6, q)$
with parameters $(v, k, r, s) = (q^6, (q^{3}+1)(q^2-1), q^2-1, -q^3+q^2-1)$,
see~\cite[\S3.3.1]{BvM}.

For $q=2$, we obtain strongly regular
graphs with parameters $(64, 9m', m',\allowbreak  m'-8)$
for some integer $m'$. Such strongly regular
graphs exist if and only if $2 \leq m' \leq 5$.

For $q=3$, we obtain strongly regular
graphs with parameters $(729, 56m', \allowbreak 2m',\allowbreak 2m'-27)$
for some integer $m'$. Such strongly regular
graphs exist if and only if $2 \leq m' \leq 12$.

\bigskip

Next we discuss $\cP = O^-(8, q)$ and $q$ even.
We obtain a strongly regular graph with parameters
\begin{align*}
  & v = q^{8},             &  & r = m(q^2-1),        \\
  & k = m(q^{4}+1)(q^2-1), &  & s = -m(q^{4}-q^2+1).
\end{align*}
For $m=1$ and $q=2$, we have a $1$-system
which would give a strongly regular graph
with parameters $(v, k, r, s) = (256, 51, 3, -13)$.
An SRG with such parameters is known, for instance
from the affine polar graph $VO^-(4, 4)$.
In general, for $q=2$ we obtain
a strongly regular graph with parameters
$(256, 51m, 3m, 3m-16)$ which are all known
to exist.

\section{Non-trivial Examples}

A \textit{nontrivial} example is
any example which we cannot reduce to an existing problem
in literature, so this is an informal term.
We have verified
for $O^+(6, 2)$, $O(7, 2)$ (and thus also for $\Sp(6,2)$), and $O^+(6, 3)$
that all regular sets are either described in
Section~\ref{sec:genexs} or disjoint unions or complements
of such examples.
Note that $O(7,2) \simeq \Sp(6,2)$ so this case is covered as well.
We have also verified that all regular sets in $\< \bm{j} \> + V_{10}$
have been described for $O^-(8, 2)$, $\Sp(6,3)$, and $O(7,3)$,
and that all regular sets in $\< \bm{j} \> + V_{11}$
have been described for $O^-(8, 2)$.

To summarize, for a rank 3 polar space $\cP$ with parameter $e$ the examples are as follows:
\begin{itemize}
 \item As shown in \S\ref{subsec:subrank3},
       if $\cP$ contains an embedded polar space $\cP'$
       of rank 3 with parameter $e-1$,
       the set of lines contained in $\cP'$
       lies in $\< \bm{j} \> + V_{10}$.
       These examples occur when $\cP$ is
       $\Sp(6,q)$ with $q$ even,
       $O(7,q)$, $U(7,q)$, or $O^-(8,q)$.
 \item As shown in \S\ref{subsec:gq},
       if $\cP$ contains an embedded generalized quadrangle $\cP'$
       with parameter $e+1$, the set of lines of $\cP'$
       lies in $\< \bm{j} \> + V_{11}$.
       Such an embedded GQ exists in $O^+(6,q)$, $U(6,q)$, and $O(7,q)$ (and so also in $\Sp(6,q)$ when $q$ is even).
 \item As shown in \S\ref{subsec:ovoids},
       if $\cP$ contains an ovoid $\cO$, we can take
       the collection of point pencils (of lines)
       through each point of $\cO$.
       This line set lies in $\< \bm{j} \> + V_{11}$.
       These examples occur in $O^+(6,q)$ for all $q$, and in $O(7,q)$ when $q = 3^h$.
 \item As shown in \S\ref{subsec:movoid},
       if $\cP$ contains an embedded GQ $\cP'$ with parameter $e+1$,
       and $\cP'$ contains an $m$-ovoid $\cO$,
       then the set of lines meeting $\cP'$ in
       precisely one point of $\cO$ lies in $\< \bm{j} \> + V_{11}$.
       This construction can be realized from known $m$-ovoids in
       $O^-(6,q) \leq O(7,q)$
       and in $O(5,q) \leq O^+(6,q)$.
 \item As shown in \S\ref{subsec:spreads},
       if $\cP$ contains a plane spread $\cS$,
       the set of all lines contained in an element of $\cS$
       lies in $\< \bm{j} \> + V_{20}$.
       Plane spreads exist in $\Sp(6,q)$ and $O^-(8,2)$,
       and thus also for $O(7,q)$ when $q$ is even.
 \item As shown in \S\ref{subsec:gh},
       if $\cP = O(7,q)$ (or $\Sp(6,q)$ with $q$ even),
       then the set of lines of an embedded split Cayley hexagon
       lies in $\< \bm{j} \> + V_{20}$.
\end{itemize}

We also obtain a few interesting examples as unions or complements of the above. Namely,
\begin{itemize}
 \item Let $g(q)$ denote the size of the largest partial partition
       of $O^+(6, q)$ into pairwise disjoint copies of $O(5, q)$.
       We find $g(2) = 7$, that is a partition.
       We find $g(3) = 7$ and $g(4) = 11$, that is $g(q) = (q^2+q+2)/2$ for $q \in \{ 3, 4\}$.
       We find $g(5) = 21$.
       Taking the lines in such a collection of disjoint copies of $O(5,q)$
       gives a line set in $\< \bm{j} \> + V_{11}$.
 \item In $O^+(6, 3)$ there exists an ovoid  $\cO$ disjoint from
       a copy of $O(5, 3)$. Hence, we can take the union of
       the lines incident with a point of $\cO$ and the lines of
       the $O(5,3)$ to get a set of lines in $\< \bm{j} \> + V_{11}$.
       It would be interesting
       to investigate when this can occur.
\end{itemize}

\begin{Problem}
 Determine $g(q)$ for all $q$.
\end{Problem}

Line sets in $\< \bm{j} \> + V_{11}$ are of particular interest,
as they can in some cases be used to construct examples of
Cameron--Liebler sets of generators~\cite{DBDR2023}.
We give here a simplified version of these results,
as they pertain to rank 3 polar spaces.

A set of lines $Y$ in a rank 3 polar space $\cP$ is called an $(m,2)$-ovoid
if every plane of $\cP$ contains precisely $m$ lines of $Y$,
that is, if $Y$ is a combinatorial design with respect to planes,
and as noted in \S\ref{subsec:comb_designs}, such a set is contained in
$\< \bm{j} \> + V_{11} + V_{21}$.
An $(m,2)$-ovoid $\cO$ is called \emph{regular} if it is a regular set of lines, that is, if it is contained in $\< \bm{j} \> + V_{11}$
(since by Proposition~\ref{prop:nothing_in_V21} there are no sets in $\< \bm{j} \> + V_{21}$).

\begin{Theorem}
 Let $\cP$ be the rank 5 polar space $O^+(10,q)$, $U(10,q)$, or $O(11,q)$
 and $\cP'$ be an embedded rank 3 polar space $O(7,q)$, $U(7,q)$, or $O^-(8,q)$, respectively,
 containing a regular $(m,2)$-ovoid $\cO$.
 Let $\cL$ be the set of generators of $\cP$ that meet $\cP'$ precisely in an element of $\cO$.
 Then $\cL$ is a Cameron--Liebler set of generators in $\cP$ with parameter $mq^{e+1}(q-1)$.
\end{Theorem}

The Cameron--Liebler sets that arise when $\cO$ is an embedded generalized quadrangle of $\cP'$ are trivial,
and the examples arising when $\cO$ comes from a hemisystem in $O^-(6,q) \subseteq O(7,q) \subseteq O^+(10,q)$
are already described in~\cite{DBDR2023}, but we do get a new family of Cameron--Liebler sets
from the $(q+1,2)$-ovoids that are constructed from an ovoid of $O(7,q) \subseteq O^+(10,q)$ when $q=3^h$.
\begin{Corollary}
 The polar space $O^+(10,q)$ contains a Cameron--Liebler set of generators
 with parameter $(q^2 -1)q$ whenever $q = 3^h$.
\end{Corollary}

\appendix

\section{\texorpdfstring{The $Q$-matrix and the vectors $aQ$}{The Q-matrix and the vectors aQ}}
\label{appendix}

We recommend the reader to enter the eigenvalue matrix $P$ into a suitable computer algebra
system and then calculate the dual eigenvalue matrix $Q$ as $nP^{-1}$ there and
then check all our claims about $aQ$ by computer.
We believe this to be the quickest option.
Here are some alternatives which should be feasible by hand:
The reader can obtain $Q$ using $Q = \Delta_P^{-1} P^T \Delta_Q$,
where $\Delta_P$ and $\Delta_Q$ are diagonal matrices having the first row
of $P$, respectively, $Q$ as their entries.
The reader may also use the explicit $Q$-matrix from below.
After we have stated the $Q$ matrix, we will also give the vectors $b = aQ$
which occur in \S\ref{sec:genexs}. Their correctness can be verified using $bP = na$.

\bigskip

The matrix $Q$ is
{
\[\begin{pmatrix}
1 & \frac{{{q}^{e{+}1}} \vartheta  \left( {{q}^{e{+}1}}{+}1\right) }{{{q}^{e}}{+}q} & \frac{{{q}^{2}} \left( q{+}1\right)  \nu }{{{q}^{e}}{+}q} & \frac{{{q}^{2 e{+}1}} \vartheta  \nu }{{{q}^{e}}{+}{{q}^{2}}} & \frac{{{q}^{e{+}3}} \vartheta  \nu }{{{q}^{e}}{+}{{q}^{2}}}\\
1 & \frac{{{q}^{e}} \vartheta \eta \left( {{q}^{e{+}1}}{+}1\right)   }{\left( q{+}1\right)  \left( {{q}^{e}}{+}1\right)  \left( {{q}^{e}}{+}q\right) } & {-} \frac{q \nu }{{{q}^{e}}{+}q}  & \frac{\left( q{-}1\right)  {{q}^{2 e}} \vartheta  \nu }{\left( {{q}^{e}}{+}1\right)  \left( {{q}^{e}}{+}{{q}^{2}}\right) } & {-} \frac{{{q}^{e{+}2}} \vartheta  \nu }{\left( q{+}1\right)  \left( {{q}^{e}}{+}{{q}^{2}}\right) } \\
1 & \frac{\vartheta  \left( {{q}^{e{+}1}}{-}1\right)  \left( {{q}^{e{+}1}}{+}1\right) }{\left( q{+}1\right)  \left( {{q}^{e}}{+}q\right) } & \frac{\left( {{q}^{2}}{+}1\right)  \nu }{{{q}^{e}}{+}q} & {-} \frac{{{q}^{e}} \vartheta  \nu }{{{q}^{e}}{+}{{q}^{2}}}  & \frac{q \vartheta  \left( {{q}^{e}}{-}q\right)  \nu }{\left( q{+}1\right)  \left( {{q}^{e}}{+}{{q}^{2}}\right) }\\
1 & \frac{q \vartheta \eta \left( {{q}^{e{+}1}}{+}1\right)   }{q \left( q{+}1\right)  \left( {{q}^{e}}{+}1\right)  \left( {{q}^{e}}{+}q\right) } & {-} \frac{q \nu }{{{q}^{e}}{+}q}  & {-} \frac{\left( q{-}1\right)  {{q}^{e{-}1}} \vartheta  \nu }{\left( {{q}^{e}}{+}1\right)  \left( {{q}^{e}}{+}{{q}^{2}}\right) }  & \frac{q \vartheta  \nu }{\left( q{+}1\right)  \left( {{q}^{e}}{+}{{q}^{2}}\right) }\\
1 & {-} \frac{\vartheta  \left( {{q}^{e{+}1}}{+}1\right) }{q \left( {{q}^{e}}{+}q\right) }  & \frac{\left( q{+}1\right)  \nu }{q \left( {{q}^{e}}{+}q\right) } & \frac{\vartheta  \nu }{q \left( {{q}^{e}}{+}{{q}^{2}}\right) } & {-} \frac{\vartheta  \nu }{q \left( {{q}^{e}}{+}{{q}^{2}}\right) } \end{pmatrix},\]
}
where $\vartheta = q^2+q+1$, $\eta = {{q}^{e{+}1}}{+}{{q}^{2}}{+}q{-}1$, and $\nu = {{q}^{e{+}2}}{+}1$.

\bigskip

\noindent
For $a$ in \S\ref{subsec:planes},
{\footnotesize
\[
 aQ = \left(q^2+q+1,	\frac{q^{e+1}(q+1)(q^2+q+1)(q^{e+1}+1)}{q^e+1},	0,	\frac{q^{2e+1}(q^2+q+1)(q^{e+2}+1)}{q^e+1},	0\right).
\]
}
For $a$ in \S\ref{subsec:ptpen},
{\footnotesize
\[
 aQ = \left((q{+}1)(q^{e+1}{+}1),	\frac{q^{e+2}(q^2{+}q{+}1)(q^e{+}1)(q^{e+1}{+}1)}{q^e{+}q},	\frac{q^3(q^{e+1}{+}1)(q^{e+2}{+}1)}{q^e{+}q},	0,	0\right).
\]
}
For $a'$ in \S\ref{subsec:ptpen},
{\footnotesize
\begin{align*}
 a'Q = \left(  q^2(q^e+1)(q^{e+1}+1) \vphantom{\frac{q^e (q+1){(q-1)}^2(q^2+q+1)(q^{e+1}+1)}{q^e+q}} \right.,	& \frac{q^e (q+1){(q-1)}^2(q^2+q+1)(q^{e+1}+1)}{q^e+q},
 \\ & \left. \frac{q(q+1)(q^e+1)(q^{e+1}+1)(q^{e+2}+1)}{q^e+q},	0,	0\right).
\end{align*}
}
For $a$ in \S\ref{subsec:systems},
{\footnotesize
\begin{align*}
 aQ = \left( q^{e+2}+1,	0, \vphantom{\frac{q^{e+1}(q^2+q+1)(q^e+1)(q^{e+2}+1)}{q^e+q^2}} \right.	& q(q+1)(q^{e+2}+1),	\frac{q^{e+1}(q^2+q+1)(q^e+1)(q^{e+2}+1)}{q^e+q^2},\\
 & \left. \frac{q^{e+1}(q^2-1)(q^2+q+1)(q^{e+2}+1)}{q^e+q^2} \right).
\end{align*}
}
For $a$ in \S\ref{subsec:subrank3},
{
\begin{align*}
 aQ = \left( (q^2+q+1)(q^e+1)(q^{e+1}+1),	q^e(q^2-1)(q^2+q+1)(q^{e+1}+1),	0,	0,	0\right).
\end{align*}
}
For $a$ in \S\ref{subsec:gq},
{
\begin{align*}
 aQ = \left( (q^{e+1}+1)(q^{e+2}+1),	0,	q(q+1)(q^{e+1}+1)(q^{e+2}+1),	0,	0\right).
\end{align*}
}
For $a$ in \S\ref{subsec:spreads},
{
\begin{align*}
 aQ = \left( (q^2+q+1)(q^{e+2}+1),	0,	0,	q^{e+1}(q^2+q+1)(q^{e+2}+1),	0\right).
\end{align*}
}
For $a$ in \S\ref{subsec:gh},
{
\begin{align*}
 aQ = \left( (q^2+q+1)(q^{3}+1),	0,	0,	q^{2}(q^2+q+1)(q^{3}+1),	0\right).
\end{align*}
}

\bibliographystyle{plain}
\bibliography{regular_line_set_rank_3}

\end{document}